\documentclass[11pt]{amsart}
\usepackage{extsizes}
\usepackage{amsfonts, mathrsfs,xcolor}
\usepackage{amssymb}
\usepackage{latexsym}
\usepackage{color, xcolor}
\usepackage{hyperref}
\usepackage{soul}
\usepackage [latin1]{inputenc}
\usepackage[french,english]{babel}

\input{xypic}
\xyoption{all}

\usepackage{graphicx,amsmath,amssymb,amsthm,array,enumerate,mhsetup,mathtools,comment,multicol,tabularx, dutchcal}

\newcommand{\authorfootnotes}{\renewcommand\thefootnote{\@fnsymbol\c@footnote}}%
\makeatother



\newtheorem{thm}{Theorem}[section]
\newtheorem{lemma}[thm]{Lemma}
\newtheorem{proposition}[thm]{Proposition}
\newtheorem{corollary}[thm]{Corollary}

\newtheorem{remark}[thm]{Remark}

\numberwithin{equation}{section}

\newcommand{\norm}[1]{\ensuremath{N(\ideal{#1})}}

\begin{document}

\def\q{\mathfrak{q}}
\def\p{\mathfrak{p}}
\def\l{\mathfrak{l}}
\def\u{\mathfrak{u}}
\def\a{\mathfrak{a}}
\def\b{\mathfrak{b}}
\def\m{\mathfrak{m}}
\def\n{\mathfrak{n}}
\def\r{\mathfrak{r}}
\def\c{\mathfrak{c}}
\def\d{\mathfrak{d}}
\def\e{\mathfrak{e}}
\def\k{\mathfrak{k}}
\def\z{\mathfrak{z}}
\def\h{{\mathfrak h}}
\def\gl{\mathfrak{gl}}
\def\sl{\mathfrak{sl}}

\def\Ext{{\rm Ext}}
\def\Hom{{\rm Hom}}
\def\Ind{{\rm Ind}}

\def\res{\mathop{Res}}

\def\GL{{\rm GL}}
\def\SL{{\rm SL}}
\def\SO{{\rm SO}}
\def\O{{\rm O}}

\def\R{\mathbb{R}}
\def\C{\mathbb{C}}
\def\Z{\mathbb{Z}}
\def\N{\mathbb{N}}
\def\Q{\mathbb{Q}}
\def\A{\mathbb{A}}
\def\D{\mathbb{D}}
\def\Re{\text{Re}}
\def\Im{\text{Im}}

\def\w{\wedge}

\def\Cat{\mathcal{C}}
\def\HC{{\rm HC}}
\def\HCat{\Cat^\HC}
\def\proj{{\rm proj}}

\def\to{\rightarrow}
\def\To{\longrightarrow}

\def\1{1\!\!1}
\def\dim{{\rm dim}}

\def\th{^{\rm th}}
\def\isom{\approx}

\def\CE{\mathcal{C}\mathcal{E}}
\def\E{\mathcal{E}}

\def\dis{\displaystyle}
\def\f{{\bf f}}                 
\def\g{{\bf g}}
\def\T{{\rm T}}              
\def\omegatil{\tilde{\omega}}  
\def\H{\mathcal{H}}            
\def\Dif{\mathfrak{D}}      
\def\W{W^{\circ}}           
\def\Whit{\mathcal{W}}      
\def\ringO{\mathcal{O}}     
\def\S{\mathcal{S}}      
\def\M{\mathcal{M}}      
\def\K{{\rm K}}          
\def\h{\mathfrak{h}} 
\def\norm{{\rm N}}       
\def\trace{{\rm Tr}} 
\def\ctilde{\tilde{C}}

\title{A Note on Large Sums of Divisor-Bounded Multiplicative Functions}

\author{Claire Frechette}
\address[Claire Frechette]{Boston College, Department of Mathematics, Chestnut Hill, MA 02467, USA}
\email{frechecl@bc.edu}

\author{Mathilde Gerbelli-Gauthier}
\address[Mathilde Gerbelli-Gauthier]{McGill University, Department of Mathematics, Montr\'eal, QC H3A 2K6, Canada}
\email{mathilde.gerbelli-gauthier@mcgill.ca}

\author{Alia Hamieh}
\address[Alia Hamieh]{University of Northern British Columbia, Department of Mathematics and Statistics, Prince George, BC V2N 4Z9, Canada}
\email{alia.hamieh@unbc.ca}

\author{Naomi Tanabe}
\address[Naomi Tanabe]{Bowdoin College,  Department of Mathematics, Brunswick, ME 04011, USA}
\email{ntanabe@bowdoin.edu}


\date{\today}
\thanks{The research of Claire Frechette is supported by NSF grant DMS-2203042, and the research of Alia Hamieh is  supported by NSERC Discovery grant RGPIN-2018-06313.}

\begin{abstract}
Given a multiplicative function $f$, we let $S(x,f)=\sum_{n\leq x}f(n)$ be the associated partial sum. In this note, we show that lower bounds on partial sums of divisor-bounded functions result in lower bounds on the partial sums associated to their products.  More precisely, we let $f_j$, $j=1,2$ be such that $|f_j(n)|\leq \tau(n)^\kappa$ for some $\kappa\in\mathbb{N}$, and assume their partial sums satisfy $\left|S(x_j,f_j)\right|\geq \eta x_j (\log x_j)^{2^\kappa-1}$ for some $x_1, x_2\gg 1$ and $\eta>\max_j\{(\log x_j)^{-1/100}\}$.
We then show that there exists $x\geq \min\{x_1, x_2\}^{\xi^2}$ such that 
$\left|S(x,f_1f_2)\right|\geq \xi x (\log x)^{2^{2\kappa}-1}$, where $\xi=C\eta^{1+2^{\kappa+3}}$ for some absolute constant $C>0$.

\end{abstract}

\maketitle

\section{Introduction}

A central theme in analytic number theory is studying mean values of multiplicative functions. One enduring quest in this area, which has garnered extensive research over the past century and continues to attract significant interest to date, revolves around understanding the asymptotic behaviour of character sums. Let $\chi$ be a primitive Dirichlet character modulo $q$ and let $S(x,\chi)=\sum_{n\leq x}\chi(n)$ be the associated character sum. Trivially, we have the bound \[|S(x,\chi)|\leq \min(x,q).\] 
An important result proven independently by P\'olya and Vinogradov \cite[pages~135--137]{davenport} asserts that 
\[\max_{x}\left|S(x,\chi)\right|\ll \sqrt{q}\log q,\] 
which was further improved to 
\[\max_{x}\left|S(x,\chi)\right|\ll \sqrt{q}\log\log q,\] 
by assuming the Generalized Riemann Hypothesis (GRH) for Dirichlet $L$-functions thanks to the work of Montgomery and Vaughan \cite{montgomery-vaughan-paper}. It follows that 
\begin{equation}\label{eqn:o(x)}S(x,\chi)=o(x)\end{equation} 
when $x\geq q^{\frac12+\epsilon}$. 
In 1957, Burgess established \eqref{eqn:o(x)} in the range $x>q^{\frac14+o(1)}$ for any quadratic character~$\chi$ when $q$ is prime. This was generalized by Burgess himself to any non-principal character provided that $q$ is cubefree, and in the smaller range $x>q^{\frac38+o(1)}$ otherwise. Burgess's range has not been substantially improved over the last five decades, although it is widely believed that \eqref{eqn:o(x)} should hold in the wider range $x\gg_{\epsilon}q^{\epsilon}$. In 2001, Granville and Soundararajan  \cite{GS1}  proved that, for  a primitive character $\chi\mod q$, \eqref{eqn:o(x)} holds in the range $\log x/\log\log q\to\infty$ assuming the GRH for $L(s,\chi)$.  In~\cite{GS3}, they also showed that this asymptotic holds under the weaker assumption that ``100\%" of the zeros of $L(s,f)$ up to height $\frac14$ lie on the critical line. This work of Granville and Soundararajan is of particular interest to us because it makes strong connections between character sums and zeros of the associated $L$-functions using general results from multiplicative number theory pertaining to mean values of 1-bounded multiplicative functions in the pretentious framework. 
In \cite{GS3}, the authors record various other interesting observations on large character sums applying variants of Halasz's theorem and Lipschitz estimate as developed in \cite{GS2}. For example, they show that if $\chi_1$ and $\chi_2$ have large character sums then so does their product $\chi_1\chi_2$. This result follows as a direct application of  \cite[Theorem 6.2]{GS3} which asserts that if the partial sums of two completely multiplicative 1-bounded functions are large, then the partial sum of their product is large as well. In this note, we extend this theorem to multiplicative functions valued outside of the unit disk, but whose magnitude is bounded by a power of the divisor function.

\begin{thm}
	\label{thm:main1}
Let $\kappa$ be a positive integer. Let $f_1$ and $f_2$ be multiplicative functions with $|f_j(n)|\leq \tau(n)^\kappa$ for all $n\in \N$. Suppose that, for $x_1, x_2\gg 1$, there exists $\eta>\max_j\{(\log x_j)^{-1/100}\}$ such that
\begin{equation}\label{eqn:eta_sum} 
\left|\sum_{n\leq x_j}f_j(n)\right|\geq \eta x_j (\log x_j)^{2^\kappa-1}.
\end{equation}
Then, with $\xi=C\eta^{1+2^{\kappa+3}}$ for some absolute constant $C>0$, there exists $x\geq \min\{x_1, x_2\}^{\xi^2}$ such that 
\[ \left|\sum_{n\leq x}f_1(n)f_2(n)\right|\geq \xi x (\log x)^{2^{2\kappa}-1}.\]
\end{thm}

Theorem \ref{thm:main1} is proved following the general structure of the argument of \cite[Theorem 6.1]{GS3}, but we extend it to divisor-bounded multiplicative functions. The technical results employed in the proofs of the various theorems in \cite{GS3} have their roots in the work of Hal\'asz \cite{halasz1,halasz2} and the subsequent works of Montgomery \cite{montgomery}, Tenenbaum \cite{Tenenbaum}, Granville and Soundararajan \cite{GS2}, among other mathematicians, on mean values of multiplicative functions that take their values in the complex unit disc. These results have been extended to a broad class of divisor-bounded multiplicative functions by Granville, Harper and  Soundararajan in \cite{GHS1}, Mangerel in \cite{MangerelThesis,Mangerel}, and Matthiesen in \cite{matthiesen}. Our work in the present paper and our previous work \cite{fght1} hinges upon such generalizations and is motivated by trying to study a $\mathrm{GL}_{2}$-analogue of \eqref{eqn:o(x)}. In fact, for  a primitive cusp form $g(z)=\sum_{n\geq1}\lambda_g(n)n^{\frac{k-1}{2}}e^{2\pi i n z}$ in $S_{k}(1)$, Lamzouri \cite[Corollary~1.2]{Lamzouri} proved that \begin{equation}\label{eqn:o(xlogx)}\sum_{n\leq x}\lambda_g(n)=o(x\log x)\end{equation} holds in the range $\log x/\log\log k\to\infty$ assuming the GRH for $L(s,g)$. He also proved unconditionally that this range in $x$ is best possible \cite[Corollary 1.4]{Lamzouri}. In \cite{fght1}, we extend the methods of \cite{GS3} to show that \eqref{eqn:o(xlogx)} holds in the range $x\geq k^{\epsilon}$ under a weaker assumption than GRH. In connection with the main result of this paper, we  derive the following immediate consequence for large sums of Fourier coefficients of primitive cusp forms, which can be viewed as a $\GL_2$-analogue of \cite[Corollary~1.7]{GS3}.

\begin{corollary}\label{cor:main}
Let $g\in S_k(r)$ and $h\in S_{\ell}(q)$ be  primitive cusp forms whose Fourier coefficients are denoted by $\lambda_g(n)$ and $\lambda_h(n)$, respecctively. Suppose that 
\[\left|\sum_{n\leq x_1}\lambda_{g}(n)\right|\geq \eta x_1\log x_1\quad \text{and}\quad  \left|\sum_{n\leq x_2}\lambda_{h}(n)\right|\geq \eta x_2\log x_2,\] for some $x_1,x_2\gg1$
and $\eta>\max_j\{(\log x_j)^{-1/100}\}$. Then there exists $x\geq (\min(x_1,x_2))^{\xi^2} $ such that
\[\left|\sum_{n\leq x_1}\lambda_{g}(n)\lambda_{h}(n)\right|\geq \xi x(\log x)^3,\] where $\xi=C\eta^{17}$ for some absolute positive constant $C$. \end{corollary}

The paper is organized as follows. In Section \ref{sec: Setting and Preliminary Results}, we record versions of Hal\'asz's Theorem and Lipschitz Formula that are suitable for  our setup. We also prove some consequences of these results that are crucial for the proof of our main theorem. In Section \ref{sec;technical}, we introduce a distance function for divisor-bounded multiplicative functions, for which we prove a triangle inequality-type result. We also derive a lower bound on sums of multiplicative functions in terms of the distance function. These results form the main ingredients of the proof of Theorem \ref{thm:main1}, which appears in Section \ref{sec:proof}.

 \vskip .1in
\noindent {\bf{Conventions and Notation.}} In this work, we adopt the following conventions and notation. Given two functions $f(x)$ and $g(x)$ we write $f(x) \ll g(x)$, $g(x) \gg f(x)$ or $f(x) = O(g(x))$ to mean there exists some positive constant $M$ such that $|f(x)| \leq M |g(x)|$ for $x$ large enough. The notation $f(x) \asymp g(x)$ is used when both estimates $f(x) \ll g(x)$ and $f(x) \gg g(x)$ hold simultaneously. We write $f(x) = o(g(x))$ when $g(x) \neq 0$ for
sufficiently large $x$ and $\displaystyle{\lim_{x\to\infty} \frac{f(x)}{g(x)} = 0}$.  The letter $p$ will be exclusively used to represent a prime number.

 \section{Hal\'asz's Theorem and Consequences } \label{sec: Setting and Preliminary Results}
 Let $f(n)$ be a multiplicative function satisfying 
 \begin{equation}\label{eqn:divisor-bounded} |f(n)|\leq \tau(n)^{\kappa}\end{equation}
 for all $n\in \N$, and let $L(s, f)$ be its associated Dirichlet $L$-series given by 
 \[ L(s, f)=\sum_{n\geq 1} \frac{f(n)}{n^s}, \quad \Re(s)>1.\]

We begin with presenting slightly modified versions of Hal\'asz's Theorem and its corollary, following the ideas of  \cite[Theorem~1.1, Corollary~1.2]{GHS1}.  
 
\begin{thm} \label{thm:Halasz} Let $f$ be a multiplicative function such than $|f(n)| \leq \tau(n)^\kappa$ for all $n\in \N$. Let $x$ be sufficiently large. Then for any $1\ll T\leq x^{\frac{9}{10}}$, we have \[
\frac{1}{x}\sum_{n\leq x} f(n) \ll \frac{1}{\log x}\int_{\frac{1}{\log x}}^1\left(\max_{|t|\leq T}\left|\frac{L(1+\sigma+it,f)}{1+\sigma+it}\right|\right)\frac{d\sigma}{\sigma} +O\left(\frac{(\log T)^{2^\kappa}}{\log x}+\frac{(\log x)^{2^\kappa}}{T}\right).
\]
\end{thm}
\begin{proof}
The result follows from  \cite[Eqs.~2.1,~2.2,~\&~2.9]{GHS1} with $y=T^2$.
\end{proof}

To further proceed, let $1\leq T_0\ll (\log x)^{2^\kappa}$, and let $M_f(x,T_0)$ be a real number satisfying
 \begin{equation}\label{eqn:M} 
\max_{|t|\leq T_0}\left|\frac{L\left(1+\frac{1}{\log x}+it, f\right)}{1+\frac{1}{\log x}+it} \right|=e^{-M_f(x,T_0)}(\log x)^{2^\kappa}.
 \end{equation}
 Then, we obtain the following corollary. 
\begin{corollary}\label{Halasz-M}
Suppose $f$ is a multiplicative function such that $|f(n)| \leq \tau(n)^\kappa$ for all $n\in \N$. Let $(\log x)^{2^{\kappa-1}}\ll T_0\ll (\log x)^{2^\kappa}$ for some sufficiently large $x$, and let $M_f(x, T_0)$ be the quantity defined in \eqref{eqn:M}. Then, we have
\[
\frac{1}{x}\sum_{n\leq x} f(n) \ll (M_f(x, T_0)+1)e^{-M_f(x, T_0)} (\log x)^{2^\kappa-1} + \frac{(\log x)^{2^\kappa-1}}{T_0}+\frac{(\log\log x)^{2^\kappa}}{\log x}.
\]
\end{corollary}
\begin{proof}
We apply Theorem \ref{thm:Halasz} with $T=T_0\log x$ to get 
 \[
\frac{1}{x}\sum_{n\leq x} f(n) \ll \frac{1}{\log x}\int_{\frac{1}{\log x}}^1\left(\max_{|t|\leq T}\left|\frac{L(1+\sigma+it,f)}{1+\sigma+it}\right|\right)\frac{d\sigma}{\sigma} +O\left(\frac{(\log x)^{2^\kappa-1}}{T_0}+\frac{(\log \log x)^{2^\kappa}}{\log x}\right).\]
Noting that $|f(n)|\leq \tau(n)^\kappa=\tau_2(n)^\kappa\leq \tau_{2^\kappa}(n)$, where $\tau_{\ell}$ is the $\ell$-fold divisor function (i.e., $\tau_{\ell}(n)$ is the coefficient of $n^{-s}$ in the Dirichlet series representation of $\zeta(s)^{\ell}$), and that $\sigma\leq 1$, we have
\[ L(1+\sigma+it, f)\ll \zeta(1+\sigma)^{2^\kappa}\ll \left(\frac{1}{\sigma}\right)^{2^\kappa},\]
and therefore,
\[\max_{|t|\leq T}\left|\frac{L(1+\sigma+it,f)}{1+\sigma+it}\right|=\max_{|t|\leq T_0}\left|\frac{L(1+\sigma+it,f)}{1+\sigma+it}\right|+O\left(\frac{1}{\sigma^{2^{\kappa}}T_0}\right).\]
It follows that 
\[\frac{1}{\log x}\int_{\frac{1}{\log x}}^1\left(\max_{|t|\leq T}\left|\frac{L(1+\sigma+it,f)}{1+\sigma+it}\right|\right)\frac{d\sigma}{\sigma} =\frac{1}{\log x}\int_{\frac{1}{\log x}}^1\left(\max_{|t|\leq T_0}\left|\frac{L(1+\sigma+it,f)}{1+\sigma+it}\right|\right)\frac{d\sigma}{\sigma} +O\left(\frac{(\log x)^{2^\kappa-1}}{T_0}\right),\] and so 
 \[
\frac{1}{x}\sum_{n\leq x} f(n) \ll \frac{1}{\log x}\int_{\frac{1}{\log x}}^1\left(\max_{|t|\leq T_0}\left|\frac{L(1+\sigma+it,f)}{1+\sigma+it}\right|\right)\frac{d\sigma}{\sigma} +O\left(\frac{(\log x)^{2^\kappa-1}}{T_0}+\frac{(\log \log x)^{2^\kappa}}{\log x}\right).\]
Using the maximum modulus principle on the region $\{u+iv: 1+\frac{1}{\log x}\leq u\leq 2,\; |v|\leq T_0\}$, the rest of the proof follows by applying the same argument employed in the proof of \cite[Corollary~1.2]{GHS1} with very minor changes. We remark here that  the lower bound $T_0\gg (\log x)^{2^{\kappa-1}}$ is imposed so that we could use \cite[Lemma~2.7]{GHS1} directly. In practice, we will apply this result with $T_0\asymp (\log x)^{2^{\kappa}}$.
\end{proof}

Let $\phi_f(x, T_0)$ be a real number in the range $|t|\leq T_0$ for which the function 
 \[t\mapsto \left|L\left(1+\frac{1}{\log x}+it, f\right)\right|\]
 attains its maximum, and let $N_f(x, T_0)$ be a real number satisfying
 \begin{equation}\label{eqn:N} 
\left|L\left(1+\frac{1}{\log x}+i\phi_f(x, T_0), f\right)\right| =e^{-N_f(x, T_0)}(\log x)^{2^\kappa}.
 \end{equation}
 As mentioned in the proof of Corollary~\ref{Halasz-M}, we will mostly be interested in the case $T_0\asymp (\log x)^{2^{\kappa}}$. 
We also note that 
\begin{align*}
e^{-M_f(x, T_0)}(\log x)^{2^\kappa}&=\max_{|t|\leq T_0}\left|\frac{L(1+\frac{1}{\log x}+it,f)}{1+\frac{1}{\log x}+it}\right| 
&\leq \max_{|t|\leq T_0}\left|L(1+\frac{1}{\log x}+it,f)\right| 
&=e^{-N_{f}(x,T_0)}(\log x)^{2^\kappa}.
\end{align*}
It then follows that $M_f(x, T_0)\geq N_f(x, T_0)$, and so $(M_f(x, T_0)+1)e^{-M_f(x, T_0)}\leq (N_f(x, T_0)+1)e^{-N_f(x, T_0)}$. Thus, we arrive at the following result.
\begin{corollary}\label{Halasz-L}
Suppose $f$ is a multiplicative function such that $|f(n)| \leq \tau(n)^\kappa$ for all $n\in \N$. Let $(\log x)^{2^{\kappa-1}}\ll T_0\ll (\log x)^{2^\kappa}$ for some sufficiently large $x$, and let $N_f(x, T_0)$ be the quantity defined in \eqref{eqn:N}. Then, we have
\[
\frac{1}{x}\sum_{n\leq x} f(n) \ll (N_f(x, T_0)+1)e^{-N_f(x, T_0)} (\log x)^{2^\kappa-1} + \frac{(\log x)^{2^\kappa-1}}{T_0}+\frac{(\log\log x)^{2^\kappa}}{\log x}.
\]
\end{corollary}

We now generalize \cite[Theorem~2a]{GS2}, which will become useful in the proof of Proposition~\ref{mvfbound} at the end of this section. 
\begin{proposition}\label{GS-decay-thm2a}
Suppose $f$ is a multiplicative function such that $|f(n)| \leq \tau(n)^\kappa$ for all $n\in \N$. Let $\phi=\phi_f(x,(\log x)^{2^\kappa})$ as in \eqref{eqn:N} and assume that $|\phi|\gg (\log x)^{2^{\kappa-1}}$. Then
\[\frac{1}{x}\sum_{n\leq x} f(n) \ll \frac{(\log x)^{2^{\kappa}-1}}{|\phi|-2}+\frac{\left(\log\log x\right)^{2^{\kappa+1}(1-\frac{2}{\pi})+1}}{(\log x)^{1-\frac{2^{\kappa+1}}{\pi}}}+\frac{(\log\log x)^{2^\kappa}}{\log x}.\]
\end{proposition}
\begin{proof}
We apply Corollary \ref{Halasz-L} with $T_0=|\phi|-2$ to get
\begin{align*}\frac{1}{x}\sum_{n\leq x} f(n) &\ll (N_f(x,|\phi|-2 )+1)e^{-N_f(x, |\phi|-2)} (\log x)^{2^\kappa-1} + \frac{(\log x)^{2^\kappa-1}}{|\phi|-2}+\frac{(\log\log x)^{2^\kappa}}{\log x}\\&
\ll \frac{1}{\log x}\max_{|t|\leq |\phi|-2}\left|L(1+\frac{1}{\log x}+it,f)\right| \log\left(\frac{(\log x)^{2^\kappa}}{\max_{|t|\leq |\phi|-2}\left|L(1+\frac{1}{\log x}+it,f)\right|}\right)\\&\hspace{2in}+\frac{(\log x)^{2^\kappa-1}}{|\phi|-2}+\frac{(\log\log x)^{2^\kappa}}{\log x}.
\end{align*}
By \cite[Lemma~4.1]{GHS1}, we have 
\begin{align*}
L\left(1+\frac{1}{\log x}+it,f\right)&\ll (\log x)^{\frac{2^{\kappa+1}}{\pi}}\left(\frac{\log x}{1+|t-\phi|\log x}+(\log\log x)^2\right)^{2^{\kappa}(1-\frac{2}{\pi})}\\
&\ll  (\log x)^{\frac{2^{\kappa+1}}{\pi}}\left(\log\log x\right)^{2^{\kappa+1}(1-\frac{2}{\pi})},
\end{align*}
for all $|t|\leq |\phi|-2$. It follows that
\begin{align*}
\frac{1}{x}\sum_{n\leq x} f(n) &\ll  
\frac{1}{\log x}(\log x)^{\frac{2^{\kappa+1}}{\pi}}\left(\log\log x\right)^{2^{\kappa+1}(1-\frac{2}{\pi})} \log\left(\frac{(\log x)^{2^\kappa}}{\max_{|t|\leq |\phi|-2}\left|L(1+\frac{1}{\log x}+it,f)\right|}\right)\\
&\hspace{2in}+\frac{(\log x)^{2^\kappa-1}}{|\phi|-2}+\frac{(\log\log x)^{2^\kappa}}{\log x}\\
&\ll \frac{\left(\log\log x\right)^{2^{\kappa+1}(1-\frac{2}{\pi})+1}}{(\log x)^{1-\frac{2^{\kappa+1}}{\pi}}}+\frac{(\log x)^{2^\kappa-1}}{|\phi|-2}+\frac{(\log\log x)^{2^\kappa}}{\log x},
\end{align*}
as required.
\end{proof}

Next, we require the following version of Lipschitz Theorem.
\begin{thm}[Lipschitz Theorem]\label{prop:lipschitz} 
Suppose $f$ is a multiplicative function such that $|f(n)|\leq \tau(n)^\kappa$. Let $\phi=\phi_f(x,(\log x)^{2^\kappa})$, as given in \eqref{eqn:N}. Then for all $1 \leq w \leq x^{1/3}$
we have
\begin{align*}
& \left| \frac{1}{x} \sum_{n \leq x}f(n)n^{-i\phi} - \frac{1}{x/w} \sum_{n \leq x/w}f(n)n^{-i\phi} \right| \\
&\hspace{2in} \ll \log \left(\frac{\log x}{\log ew}\right) \left(\frac{\log w+(\log\log x)^2}{\log x}\right)^{\min\{2^\kappa(1-\frac{2}{\pi}), 1\}}(\log x)^{2^\kappa-1}. 
\end{align*}
\end{thm}
\begin{proof} The proof is very similar to that of \cite[Theorem~1.5]{GHS1}, and the reader is referred to their detailed exposition. 
\end{proof}

\begin{corollary}\label{cor3.9-mangerel} 
 Let $f$ and $\phi$ be as in Theorem~\ref{prop:lipschitz}.  Then
\begin{equation*}\label{MangerelLemma}
\frac{1}{x}\sum_{n\leq x} f(n) = \frac{x^{i\phi}}{1+i\phi}\cdot \frac{1}{x}\sum_{n\leq x} f(n) n^{-i\phi} +O(E_\kappa(x)),\end{equation*}
where 
\begin{equation}\label{eqn:E}
 E_\kappa(x)=\begin{cases} (\log x)^{-1+4/\pi}(\log\log x)^{5-8/\pi}\,\,\, \text{ if } \kappa=1 \\
(\log x)^{2^\kappa-2}(\log\log x)^3\,\,\, \text{ if } \kappa>1. \end{cases}
\end{equation}
\end{corollary}
 \begin{proof}
We will show the equivalent statement that  
\[ \frac1x\sum_{n \leq x}f(n)n^{-i\phi} = \frac{1+i\phi}{x^{1+i\phi}}\sum_{n \leq x}f(n) + O(|\phi| E_\kappa(x)).
 \]
By partial summation, we have
\begin{align*}
\frac1x\sum_{n \leq x}f(n)n^{-i\phi} &= \frac1x\int_1^x u^{-i\phi} d\left\{ \sum_{n\leq x} f(n)\right\}= \frac{1}{x^{1+i\phi}}\sum_{n\leq x} f(n) + \frac{i\phi}{x}\int_1^x \frac{1}{u^{1+i\phi}}\sum_{n\leq u}f(n)du.
\end{align*}
We split the integral into two pieces as follows:
\begin{align*}
\int_1^x \frac{1}{u^{1+i\phi}}\sum_{n\leq u}f(n)du &= \int_1^{x/(\log x)^2} \frac{1}{u^{1+i\phi}}\sum_{n\leq u}f(n)du + \int_{x/(\log x)^2}^x \frac{1}{u^{1+i\phi}}\sum_{n\leq u}f(n)du.
\end{align*}
For the first integral, we apply the trivial bound
\begin{equation}\label{eqn:partialbound} \sum_{n\leq x} |f(n)| 
\leq \sum_{n\leq x} \tau_{2^\kappa}(n)\ll x(\log x)^{2^\kappa-1}
\end{equation}
(see \cite{luca-toth} for example) to get
\begin{align*}
\frac{i\phi}{x}\int_1^{x/(\log x)^2} \frac{1}{u^{1+i\phi}}\sum_{n\leq u}f(n)du &\ll \frac{|\phi|}{x}\int_1^{x/(\log x)^2} \frac{1}{u}\sum_{n\leq u}|f(n)|du  \\
&\leq \frac{|\phi|}{x}\int_1^{x/(\log x)^2} (\log u)^{2^\kappa-1} \, du 
\ll |\phi|E_\kappa(x).
\end{align*}
Since $w = x/u$ is in the range of Theorem~\ref{prop:lipschitz}, the second integral is equal to 
\begin{align*}
&\frac{i\phi}{x}  \int_{x/(\log x)^2}^x  \left(  \frac{1}{x^{1+i\phi}}\sum_{n\leq x}f(n) + O\left( \log \left(\frac{\log x}{\log ew}\right) \left(\frac{\log w+(\log\log x)^2}{\log x}\right)^{\min\{2^\kappa(1-\frac{2}{\pi}), 1\}}(\log x)^{2^\kappa-1}  \right)\right)du\\
 = &\frac{i\phi}{x}\left( \frac{1}{x^{1+i\phi}}\sum_{n\leq x}f(n)  \right)\int_{x/(\log x)^2}^x du +  \frac{i\phi}{x}\cdot O\left( \log\log x \left(\frac{(\log\log x)^2}{\log x}\right)^{\min\{2^\kappa(1-\frac{2}{\pi}), 1\}}(\log x)^{2^\kappa-1}\right)\int_{x/(\log x)^2}^x du\\
= & \frac{i\phi}{x^{1+i\phi}}\sum_{n\leq x}f(n) +O(|\phi| E_\kappa(x)).
\end{align*}
Combining the two integrals gives the desired result. 
\end{proof}

We now utilize Proposition~\ref{GS-decay-thm2a} (when $|\phi|$ is large) and Corollary \ref{cor3.9-mangerel} (when $|\phi|$ is small) to get the proposition that will be used in our argument.

\begin{proposition}\label{mvfbound}
Let $f$ be a multiplicative function such that $|f(n)|\leq \tau(n)^\kappa$, and let  $N=N_f(x,(\log x)^{2^\kappa})$  and $\phi=\phi_{f}(x,(\log x)^{2^\kappa})$ for sufficiently large $x$. Then,
\[
\frac{1}{x}\sum_{n\leq x} f(n) \ll (\log x)^{2^\kappa-1} \left(\frac{(N+1)e^{-N}}{1+|\phi|} + \frac{E_\kappa(x)}{(\log x)^{2^\kappa-1}}\right),
\]
where $E_\kappa(x)$ is given in \eqref{eqn:E}.
\end{proposition}
\begin{proof}
Suppose that $|\phi|\geq \frac{1}{2}(\log x)^{2^\kappa}$. By Proposition~\ref{GS-decay-thm2a}, we have \[\frac{1}{x}\sum_{n\leq x} f(n) \ll \frac{(\log x)^{2^{\kappa}-1}}{(\log x)^{2^\kappa}}+\frac{\left(\log\log x\right)^{2^{\kappa+1}(1-\frac{2}{\pi})+1}}{(\log x)^{1-\frac{2^{\kappa+1}}{\pi}}}+\frac{(\log\log x)^{2^\kappa}}{\log x}\ll \frac{\left(\log\log x\right)^{2^{\kappa+1}(1-\frac{2}{\pi})+1}}{(\log x)^{1-\frac{2^{\kappa+1}}{\pi}}}+\frac{(\log\log x)^{2^\kappa}}{\log x}.\]
It follows that 
\begin{align*}\frac{1}{x}\sum_{n\leq x} f(n) - (\log x)^{2^\kappa-1} \frac{(N+1)e^{-N}}{1+|\phi|}&\ll\frac{\left(\log\log x\right)^{2^{\kappa+1}(1-\frac{2}{\pi})+1}}{(\log x)^{1-\frac{2^{\kappa+1}}{\pi}}}+\frac{(\log\log x)^{2^\kappa}}{\log x} \ll E_{\kappa}(x).
\end{align*}
Hence, we may assume that $|\phi|\leq \frac{1}{2}(\log x)^{2^\kappa}$.
Applying Corollary~\ref{Halasz-L} to the function $g(n)=f(n)n^{-i\phi}$ with $T_0=\frac{1}{2}(\log x)^{2^\kappa}$ gives
\begin{align}\label{eqn:halasz-for-twisted-f}
\frac{1}{x}\sum_{n\leq x} f(n)n^{-i\phi}= \frac{1}{x}\sum_{n\leq x} g(n) &\ll \left(N_g\left(x,\frac{(\log x)^{2^\kappa}}{2}\right)+1\right)e^{-N_g(x,\frac{1}{2}(\log x)^{2^\kappa})} (\log x)^{2^\kappa-1} + \frac{(\log\log x)^{2^\kappa}}{\log x} \nonumber\\
&=(N+1)e^{-N} (\log x)^{2^\kappa-1} + \frac{(\log\log x)^{2^\kappa}}{\log x}.
\end{align}
The last line follows from the fact that $N_g(x,\frac{1}{2}(\log x)^{2^\kappa})= N_{f}(x,(\log x)^{2^\kappa})=N$. Indeed, this can be observed as following: Since $L(s, g)=L(s+i\phi, f)$, we have
\begin{align*} 
\max_{|t|<\frac12(\log x)^{2^\kappa}} \left|L\left(1+\frac{1}{\log x}+it, g\right)\right| &=\max_{|t|<\frac12(\log x)^{2^\kappa}} \left|L\left(1+\frac{1}{\log x}+i(t+\phi), f\right)\right| \\
&=\max_{|t|<(\log x)^{2^\kappa}} \left|L\left(1+\frac{1}{\log x}+it, f\right)\right|.
\end{align*}
Applying Corollary~\ref{cor3.9-mangerel} and \eqref{eqn:halasz-for-twisted-f}, we get
\begin{align*}
\frac{1}{x}\sum_{n\leq x} f(n) &\ll  \left|\frac{x^{i\phi}}{1 + i\phi}\right|\left((N+1)e^{-N} (\log x)^{2^\kappa-1} + \frac{(\log\log x)^{2^\kappa}}{\log x}\right) +O(E_\kappa(x)) \\
&\ll  \frac{(N+1)e^{-N}}{1 + |\phi|} ( \log x)^{2^\kappa-1} + O\left(\frac{(\log\log x)^{2^\kappa}}{\log x}\right) +  O\left(E_\kappa(x) \right).
\end{align*}
The first error term is subsumed into the second one since it is smaller. 
\end{proof}
 
 Before concluding this section, we establish the following lemma, which will play a pivotal role in Section~\ref{sec:proof}.
 
 \begin{lemma}\label{lem:M}
 Let $f$ be a multiplicative function such that $|f(n)| \leq \tau(n)^\kappa$ for all $n \in \N$.  Let  $N=N_f(x,(\log x)^{2^\kappa})$  and $\phi=\phi_{f}(x,(\log x)^{2^\kappa})$ as in \eqref{eqn:N}. Suppose that there exists $x\gg 1$ and $\eta>(\log x)^{-1/100}$ such that
 \begin{equation}\label{eqn:sum_n}
 	\left|\sum_{n\leq x}f(n)\right|\geq \eta x (\log x)^{2^\kappa-1}.
 \end{equation} Then, \[|\phi|\ll\frac{1}{\eta}\quad \text{and}\quad 
 N\leq 2 \log \left(\frac{1}{\eta}\right)+O(1).\] 
 \end{lemma}
 \begin{proof} By Proposition~\ref{mvfbound},  we know that 
 \[ \left|\sum_{n\leq x} f(n)\right| \ll x(\log x)^{2^\kappa-1} \left(\frac{(N+1)e^{-N}}{1+|\phi|} + \frac{E_\kappa(x)}{(\log x)^{2^\kappa-1}}\right).
 \]
Therefore, together with our assumption~\eqref{eqn:sum_n},
 we have
\begin{equation}\label{eqn:eta}
\eta\leq \frac{C(N+1)e^{-N}}{1+|\phi|}+\frac{CE_\kappa(x)}{(\log x)^{2^\kappa-1}},
\end{equation}
for some absolute constant $C>0$. Writing the second term on the right-hand side as $A_\kappa(x)$, the inequality~\eqref{eqn:eta} can be written as
\[ (1+|\phi|)(\eta-A_\kappa(x))\leq C(N+1)e^{-N}.\]
Since $\eta\geq 2A_\kappa(x)$ for sufficiently large $x$, we see that
 \begin{equation}\label{eqn:phi_bound} 
 1+|\phi|\leq \frac{C(N+1)e^{-N}}{\eta-A_\kappa(x)}\leq  \frac{2C(N+1)e^{-N}}{\eta}.
 \end{equation}
In particular, $|\phi|\ll \frac{1}{\eta}$ and 
$\eta\ll e^{-\frac{N}{2}}$. It follows that $N\leq 2\log\left(\frac{1}{\eta}\right)+O(1)$ as desired.
\end{proof}

\section{Technical Results}\label{sec;technical}

While the notion of a distance between multiplicative functions makes most sense in
the context of functions taking values in the complex unit disc, we adapt the standard definition (see for example \cite[Section~2]{GS3}) and set the following notation for a distance function associated with multiplicative functions $f$ satisfying \eqref{eqn:divisor-bounded}:
\begin{align} \label{distance} 
 \D^2(f,n^{it};x) = \sum_{p \leq x} \frac{2^\kappa-\Re(f(p)p^{-it})}{p}. 
 \end{align}
 The distance function  $\D^2(f,n^{it};x)$ is related to the Dirichlet series $L(s, f)$ via
 \[ \left|L\left(1+\frac{1}{\log x}+it, f\right)\right|\asymp (\log x)^{2^\kappa}\exp(-\D^2(f, n^{it}; x)), \]
 which can be observed in the following way. We have
  \begin{align}\label{eq: compare L-function and distance}
 \left|L\left(1+\frac{1}{\log x}+it, f\right)\right|&= \exp\left(\sum_{p\leq x}\frac{\Re(f(p)p^{-it})}{p}+O(1)\right)\\
 &=\exp\left(\sum_{p\leq x}\frac{2^\kappa}{p}-\sum_{p\leq x}\frac{2^\kappa-\Re(f(p)p^{-it})}{p}+O(1)\right) \nonumber  \\
 &\asymp (\log x)^{2^\kappa}\exp\left(-\sum_{p\leq x}\frac{2^\kappa-\Re(f(p)p^{-it})}{p}\right) \nonumber  \\
 &=(\log x)^{2^\kappa} \exp(-\D^2(f, n^{it}; x)), \nonumber
  \end{align}
 where the first equality follows from \cite[Lemma 2.2.15]{MangerelThesis}.
 
We note that this distance function is related to 
$N=N_f(x, (\log x)^{2^\kappa})$, defined in \eqref{eqn:N}, in the following way:
\begin{equation}\label{eqn:LvsD}
e^{-N}=\frac{1}{(\log x)^{2^\kappa}}\left|L(1+\frac{1}{\log x}+i\phi,f)\right|\asymp  \exp(-\D^2(f, n^{i\phi}; x))
.\end{equation}

It is known that the standard distance function between 1-bounded multiplicative functions satisfies a triangle inequality. While such a result does not hold in our setting, we next establish a triangle-type inequality satisfied by the function $\D(f, n^{it};x)$ we introduced in \eqref{distance}.

\begin{lemma}\label{lem:triangle_ineq}
	For $j=1,2$, let $f_j$  be a multiplicative function such that $|f_j(n)| \leq \tau(n)^\kappa$, and let $t_j$ be a real number. We have
	\[ \D(f_1, n^{it_1};x_1)+\D(f_2, n^{it_2};x_2)\geq \frac{1}{\sqrt{2^{\kappa}}}\D(f_1f_2, n^{i(t_1+t_2)}; \min\{x_1,x_2\}),\] 
where the distance function on the right-hand side is defined with respect to the condition $|f_1(n)f_2(n)|\leq \tau(n)^{2\kappa}$. 
\end{lemma}
\begin{proof} For simplicity, let us denote $\D_j= \D(f_j, n^{it_j};x_j)$ for $j=1, 2$, $x=\min\{x_1,x_2\}$, and $\D=\D(f_1f_2, n^{i(t_1+t_2)}; x)$. Then, we have
	\begin{align*}
		(\D_1+\D_2)^2&
		=\sum_{p \leq x_1} \frac{2^\kappa-\Re(f_1(p)p^{-it_1})}{p}+\sum_{p \leq x_2} \frac{2^\kappa-\Re(f_2(p)p^{-it_2})}{p}\\
		&\hspace{1in} +2\left(\sum_{p \leq x_1} \frac{2^\kappa-\Re(f_1(p)p^{-it_1})}{p}\right)^{1/2}\left(\sum_{p \leq x_2} \frac{2^\kappa-\Re(f_2(p)p^{-it_2})}{p}\right)^{1/2}\\
		&\geq \sum_{p \leq x} \frac{2^{\kappa +1}-\Re(f_1(p)p^{-it_1})-\Re(f_2(p)p^{-it_2})}{p}\\
		&\hspace{1in} +2\left(\sum_{p \leq x} \frac{2^\kappa-\Re(f_1(p)p^{-it_1})}{p}\right)^{1/2}\left(\sum_{p \leq x} \frac{2^\kappa-\Re(f_2(p)p^{-it_2})}{p}\right)^{1/2}\\
		&\geq \sum_{p\leq x}\frac{2^\kappa}{p}\left[2-\Re\left(\frac{f_1(p)}{2^\kappa}p^{-it_1}\right)-\Re\left(\frac{f_2(p)}{2^\kappa}p^{-it_2}\right)\right]\\
		&\hspace{1in}+\sum_{p\leq x}\frac{2^\kappa}{p}\sqrt{2\left(1-\Re\left(\frac{f_1(p)}{2^\kappa}p^{-it_1}\right)\right)}\sqrt{2\left(1-\Re\left(\frac{f_2(p)}{2^\kappa}p^{-it_2}\right)\right)}\\
		&\geq \sum_{p\leq x}\frac{2^\kappa}{p}\left[2-\Re\left(\frac{f_1(p)}{2^\kappa}p^{-it_1}\right)-\Re\left(\frac{f_2(p)}{2^\kappa}p^{-it_2}\right)+\left|\Im\left  (\frac{f_1(p)}{2^\kappa}p^{-it_1}\right)\Im\left(\frac{f_2(p)}{2^\kappa}p^{-it_2}\right)\right| \right], 
	\end{align*}
where the last inequality follows from $\left|\frac{f_j(p)}{2^\kappa}p^{-it_j} \right| \leq 1$ for $j=1, 2$. This fact also gives us
	\begin{align*}
		&2-\Re\left(\frac{f_1(p)}{2^\kappa}p^{-it_1}\right)-\Re\left(\frac{f_2(p)}{2^\kappa}p^{-it_2}\right)+\left|\Im\left  (\frac{f_1(p)}{2^\kappa}p^{-it_1}\right)\Im\left(\frac{f_2(p)}{2^\kappa}p^{-i\phi_2}\right)\right|   \\
		&\hspace{1in} \geq 1-\Re\left(\frac{f_1(p)}{2^\kappa}p^{-it_1}\right)\Re\left(\frac{f_2(p)}{2^\kappa}p^{-it_2}\right)+\Im\left(\frac{f_1(p)}{2^\kappa}p^{-it_1}\right)\Im\left(\frac{f_2(p)}{2^\kappa}p^{-it_2}\right)\\
		&\hspace{1in} =1-\Re\left(\frac{f_1(p)f_2(p))}{2^{2\kappa}}p^{-i(t_1+t_2})\right).
	\end{align*}
	It then follows that
	\begin{align*}
		(\D_1+\D_2)^2&\geq \sum_{p\leq x}\frac{2^\kappa}{p}\left(1-\Re\left(\frac{f_1(p)f_2(p))}{2^{2\kappa}}p^{-i(\phi_1+\phi_2})\right)\right)\\
		&=\frac{1}{2^{\kappa}}\sum_{p\leq x}\frac{2^{2\kappa}-\Re(f_1(p)f_2(p)p^{-i(\phi_1+\phi_2)})}{p}=\frac1{2^{\kappa}}\D^2,
	\end{align*}
	as desired.
\end{proof}

To conclude this section, we exhibit lower bounds of sums of multiplicative functions in terms of the distance function. More precisely, we show that if the distance  $\D(f,n^{i\psi};x)$ is small then the partial sums of $f$ get large in suitable ranges.  
  \begin{proposition}\label{prop}
 Let $f$ be a multiplicative function satisfying $|f(n)|\leq \tau(n)^\kappa$ for all $n\in \N$. Let $\psi=\psi_{f}(x,(\log x)^{2^\kappa})$ be a number in the range $|t|\leq (\log x)^{2^\kappa}$ where the maximum in \eqref{eqn:M} is attained. Let $x\gg 1$ and set $\lambda=\D^2(f, n^{i\psi}; x)+\log(1+|\psi|)+c$ where $c$ is a suitably large constant.  Then, there exists $y\in [x^\gamma, x]$, with $\gamma=1/(\lambda e^\lambda)$, such that 
 \begin{align} \label{eq: lower bound kappa} \left|\sum_{n\leq y} f(n)\right|\gg \frac{\exp(-\D^2(f, n^{i\psi}; x))}{1+|\psi|}y(\log y)^{2^\kappa-1}. \end{align}
 \end{proposition}
 \begin{proof} The strategy of the proof is to produce contradictory upper and lower bounds on the quantity 
 \[ \left|L\left(1+ \frac{2\lambda}{\log x}+i\psi,f\right)\right|, \]
under the assumption that \eqref{eq: lower bound kappa} fails. 
 	
For the  lower bound, we observe that  \begin{align}\begin{split} \label{eq: lower bound on L-function}
 		\left|L\left(1+ \frac{2\lambda}{\log x}+i\psi,f\right)\right| &= \exp \left( \sum_{p \leq x^{\frac{1}{2\lambda}}} \frac{\Re(f(p)p^{-i\psi})}{p} + O(1) \right)\\  
 		&\geq \exp \left( \sum_{p \leq x^{\frac{1}{2\lambda}}} \frac{\Re(f(p)p^{-i\psi})}{p} - \sum_{x^{\frac1{2\lambda}}<p\leq x} \frac{2^\kappa-\Re(f(p)p^{-i\psi})}{p} + O(1) \right)\\
 		&\geq \exp \left(  \sum_{p \leq x} \frac{\Re(f(p)p^{-i\psi})}{p} - \sum_{x^{\frac{1}{2\lambda} }\leq p \leq x}\frac{2^\kappa}{p}+ O(1) \right) \\ 
 		&= \left(2\lambda\right)^{-2^\kappa} \exp\left(\sum_{p \leq x} \frac{\Re(f(p)p^{-i\psi})}{p} + O(1) \right) \\
 		&= \left(2\lambda\right)^{-2^\kappa} (\log x)^{2^\kappa} \exp(-\D^2(f,n^{i\psi};x)).
\end{split}
 	\end{align}

We now move to the upper bound. Set $\delta:=\frac{2\lambda}{\log x}$. From the Mellin transform representation of Dirichlet series, we get 
 \begin{equation} \label{eqn: integral uuper bound 5.3}	\left|L\left(1+\delta+i\psi,f\right)\right| \leq |1+\delta + i\psi|\left(\int_1^\infty \frac1{y^{2+\delta}}\left|\sum_{n \leq y}f(n)\right|dy\right).  
 \end{equation} 
 Applying \eqref{eqn:partialbound}, the integral on the right-hand side of  \eqref{eqn: integral uuper bound 5.3} is bounded by a constant multiple of \[ \int_1^\infty  y^{-1-\delta} (\log y)^{2^\kappa-1}dy.\] 
Observe that  \[\int y^{-1-\delta} (\log y)^{2^\kappa-1}\; dy=  -\frac{P(\delta\log y)}{\delta^{2^\kappa} y^\delta},\]  where \[P(X) := (2^\kappa-1)!\sum_{k=0}^{2^\kappa-1}\frac{X^k}{k!}.\]
 For contradiction, we now assume that the upper bound \begin{equation} \label{eqn: false assumption upper bound 6.1} \left|\sum_{n\leq y} f(n)\right| \leq  e^{-\lambda} y(\log y)^{2^\kappa}\end{equation} holds for all $y \in [x^\gamma,x]$, 
 and split the integral into three parts, so that \begin{equation*} \label{eqn: anyiderivative three bits}  \int_1^\infty  y^{-1-\delta} (\log y)^{2^\kappa-1}dy= \left[ -\frac{P(\delta\log y)}{\delta^{2^\kappa}y^\delta} \right]^{x^\gamma}_1 + \left[ -e^{-\lambda}\frac{P(\delta\log y)}{\delta^{2^\kappa} y^\delta} \right]^x_{x^\gamma} + \left[ -\frac{P(\delta\log y)}{\delta^{2^\kappa} y^\delta} \right]^\infty_x.
 \end{equation*}  
Noting that $P(0) = (2^\kappa-1)!$ and recalling that $\delta:=\frac{2\lambda}{\log x}$, we get \begin{align*}  \int_1^\infty  y^{-1-\delta} (\log y)^{2^\kappa-1}dy&=\frac{(2^\kappa-1)!}{\delta^{2^\kappa}}
	-\frac{P(\delta \gamma \log x)}{\delta^{2^\kappa} x^{\delta\gamma}} 
	+ e^{-\lambda} \frac{P(\delta \gamma \log x)}{\delta^{2^\kappa} x^{\delta\gamma}} 
	- e^{-\lambda} \frac{P(\delta\log x)}{\delta^{2^\kappa}x^\delta}
	+\frac{P(\delta\log x)}{\delta^{2^\kappa} x^\delta} \\ 
 =&  \frac{(2^\kappa -1)!(\log x)^{2^\kappa}}{(2\lambda)^{2^\kappa}}
 -\frac{(\log x)^{2^\kappa} P(2 \gamma  \lambda )}{(2\lambda)^{2^\kappa} e^{2\lambda \gamma}}
 + e^{-\lambda} \frac{(\log x)^{2^\kappa} P( 2\gamma  \lambda )}{(2\lambda)^{2^\kappa} e^{2 \lambda \gamma}} 
\\&\hspace{1em} - e^{-\lambda} \frac{(\log x)^{2^\kappa} P(  2 \lambda )}{(2\lambda)^{2^\kappa} e^{2 \lambda }} 
 +\frac{(\log x)^{2^\kappa} P(  2 \lambda )}{(2\lambda)^{2^\kappa} e^{2 \lambda }} \\
 =& \frac{e^{-\lambda}(\log x)^{2^\kappa}}{(2 \lambda)^{2^\kappa}}\left(   (2^\kappa-1)! e^\lambda
 -\frac{ e^{\lambda} P(2 \gamma  \lambda )}{ e^{2 \lambda \gamma}}
 +  \frac{ P(2 \gamma  \lambda )}{ e^{2 \lambda \gamma}} - \frac{ P(  2\lambda )}{e^{2 \lambda }} 
 +\frac{e^{\lambda}  P(  2 \lambda ) }{ e^{2 \lambda }}  \right).
 \end{align*} 
We want to verify that the inside of the parenthesis is bounded as $\lambda \to  \infty$. Recall that $\gamma= \frac{1}{\lambda e^\lambda}$ and  $\kappa \geq 1$, so that \[ \frac{ P( 2\gamma  \lambda )}{ e^{2 \lambda \gamma}} - \frac{ P(  2 \lambda )}{e^{2 \lambda }} 
+\frac{e^{\lambda}  P(  2 \lambda ) }{ e^{2 \lambda }} \ll 1.   \] 
For the first two terms in the parenthesis, let $Q(X) := P(X) - (2^\kappa-1)!$. We have \[  (2^\kappa-1)! e^\lambda
-\frac{ e^{\lambda} P( 2\gamma  \lambda )}{ e^{2 \lambda \gamma}} = (2^\kappa-1)! e^{\lambda} (1-e^{-2 \lambda \gamma}) + \frac{e^\lambda Q(2e^{-\lambda})}{e^{2 e^{-\lambda}}} = (2^\kappa-1)! e^{\lambda} (1-e^{-2\lambda \gamma}) + O(1), \] since $Q(X)$ is a polynomial with vanishing constant term. It remains to bound $(2^\kappa-1)! e^{\lambda} (1-e^{-2 \lambda \gamma})$. The Taylor expansion of $e^X$ for $X=o(1)$ gives \[ (2^\kappa-1)! e^\lambda(1-e^{-2 \lambda \gamma}) = (2^\kappa-1)! e^\lambda( 2\lambda\gamma + O((2\lambda\gamma)^2) ) = (2^\kappa-1)! e^{\lambda}(2 e^{-\lambda} + O(e^{-2\lambda})) = O(1).  \] We conclude that 
\[ \left|L\left(1+\frac{2 \lambda }{\log x}+i\psi,f\right)\right|  \ll  \left|1+\frac{2 \lambda }{\log x} + i\psi\right| \frac{e^{-\lambda}(\log x)^{2^\kappa}}{(2\lambda)^{2^\kappa}}.\] 
From the definition of $\lambda$, we have 
\[ \lambda=\D^2(f, n^{i\psi}; x)+\log(1+|\psi|)+c = O(\log \log x), \] 
so that 
\begin{align*}
	\left|L\left(1+\frac{2 \lambda }{\log x}+i\psi,f\right)\right| & \ll \left|1+ i\psi\right| \frac{\exp(-\D^2(f, n^{i\psi}; x))\exp(- \log \left|1+ i\psi\right|-c)(\log x)^{2^\kappa}}{(2\lambda)^{2^\kappa}} \\ 
	&=  \frac{\exp(-\D^2(f, n^{i\psi}; x))\exp(-c)(\log x)^{2^\kappa}}{(2\lambda)^{2^\kappa}}. 
\end{align*}
Comparing the upper bound from this last line with the conclusion of \eqref{eq: lower bound on L-function}, we get 
\[   \frac{\exp(-\D^2(f, n^{i\psi}; x))(\log x)^{2^\kappa}}{(2\lambda)^{2^\kappa}} \ll \frac{\exp(-\D^2(f, n^{i\psi}; x))e^{-c}(\log x)^{2^\kappa}}{(2\lambda)^{2^\kappa}}, \] 
which can be made false for $c$ large enough. It follows that the assumption \eqref{eqn: false assumption upper bound 6.1} is false, and that the result holds. 	
 \end{proof}

 \vskip .2in 
 \section{Proof of Theorem~\ref{thm:main1}} \label{sec:proof}
Now, we will prove the main theorem. For $j=1,2$, we set $\phi_j=\phi_{f_j}(x_j,(\log x_j)^{2^\kappa})$ and $N_j=N_{f_j}(x_j,(\log x_j)^{2^\kappa})$ as in \eqref{eqn:N}.
We also set $f=f_1f_2$ and $\phi=\phi_{f}(x, (\log x)^{2^\kappa})$ with $x=\min\{x_1, x_2\}$. 
By Lemma~\ref{lem:triangle_ineq}, we have
\begin{align*}
\D^2(f_1f_2, n^{i(\phi_1+\phi_2)}; x) &\leq 2^{\kappa}\left[\D(f_1, n^{i\phi_1};x_1)+\D(f_2, n^{i\phi_2};x_2)\right]^2
\\& \leq 2^{\kappa+2}\max_j\left\{\D^2(f_j, n^{i\phi_j};x_j)\right\}.
\end{align*}
Together with \eqref{eqn:LvsD}, it follows that
\begin{equation*}
\exp\left(-\D^2(f_1f_2, n^{i(\phi_1+\phi_2)}; x)\right)\geq \min_j\left\{e^{-2^{\kappa+2}N_j}\right\}.
\end{equation*}
Applying Lemma~\ref{lem:M} yields
\begin{equation}\label{eqn:exp_D} 
\exp\left(-\D^2(f_1f_2, n^{i(\phi_1+\phi_2)}; x)\right)
\geq \left(\eta^{2}\right)^{2^{\kappa+2}}=\eta^{2^{\kappa+3}}.
\end{equation}

 On the other hand, by Proposition~\ref{prop}, there exists $y\in [x^{1/(\lambda e^\lambda)}, x]$ such that
 \begin{align}\label{eqn:f_bound}
 \left|\sum_{n\leq y} f_1(n)f_2(n)\right| &\gg \frac{\exp(-\D^2(f, n^{i\psi};x))y(\log y)^{2^\kappa-1}}{1+|\psi|} \nonumber \\
 &\gg \frac{\exp(-\D^2(f, n^{i(\phi_1+\phi_2)};x))y(\log y)^{2^\kappa-1}}{1+|\phi_1+\phi_2|}, 
 \end{align}
  where we recall that $\lambda=\D^2(f, n^{i\psi}; x)+\log(1+|\psi|)+c$ for some sufficiently large constant $c$. Observe that the last inequality in \eqref{eqn:f_bound} follows from \begin{align*}
 \frac{(\log x)^{2^\kappa}\exp(-\D^2(f, n^{i\psi_f};x))}{\left|1+\frac{1}{\log x}+i\psi_f\right|} &\asymp \frac{\left| L\left(1+\frac{1}{\log x}+i\psi_f, f\right)\right|}{\left|1+\frac{1}{\log x}+i\psi_f\right|}=\max_{|t|\leq (\log x)^{2^\kappa}} \frac{\left| L\left(1+\frac{1}{\log x}+it, f\right)\right|}{\left|1+\frac{1}{\log x}+it\right|} \\
&  \geq \frac{\left| L\left(1+\frac{1}{\log x}+i(\phi_1+\phi_2), f\right)\right|}{\left|1+\frac{1}{\log x}+i(\phi_1+\phi_2)\right|} \\
& \asymp \frac{(\log x)^{2^\kappa}\exp(-\D^2(f, n^{i(\phi_1+\phi_2)};x))}{\left|1+\frac{1}{\log x}+i(\phi_1+\phi_2)\right|}.
 \end{align*}

 Now, putting \eqref{eqn:exp_D} and \eqref{eqn:f_bound} together gives
 \begin{align*}
 \left|\sum_{n\leq y} f_1(n)f_2(n)\right| &\gg  \frac{\eta^{2^{\kappa+3}} y (\log y)^{2^{2\kappa}-1}}{1+|\phi_1+\phi_2|}. \end{align*}
 Finally, since $1+|\phi_1+\phi_2|\leq 1+|\phi_1|+|\phi_2|\ll 1/\eta$ by Lemma \ref{lem:M}, we conclude that 
 \begin{align*}
 \left|\sum_{n\leq y} f_1(n)f_2(n)\right| \gg  \eta^{1+2^{\kappa+3}} y (\log y)^{2^{2\kappa}-1},
 \end{align*}
 which completes the proof by observing that 
 \begin{align*}
 \frac{1}{\lambda e^{\lambda}}&\geq e^{-2\lambda}\\&=\left(\exp(-\D^2(f, n^{i\psi}; x)-\log(1+|\psi|)-c)\right)^2\\&\gg \left(\frac{\exp(-\D^2(f, n^{i(\phi_1+\phi_2)};x))}{\left|1+\frac{1}{\log x}+i(\phi_1+\phi_2)\right|}\right)^2\\&\gg \left(\eta^{1+2^{\kappa+3}}\right)^2,
 \end{align*} 
 and setting $\xi=C\eta^{1+2^{\kappa+3}}$, for some absolute positive constant $C$. \hfill $\square$

 \section*{Acknowledgements}
This project originated from Women~In~Numbers~6 Research Workshop that took place at Banff International Research Station in March 2023. The authors express their utmost gratitude to the organizers for the invaluable opportunity provided by the workshop. The third author would also like to thank Oleksiy Klurman for answering a question related to this work.

\bibliographystyle{siam}
\bibliography{references}

\end{document}